\documentclass{amsart}
\numberwithin{equation}{section}
\newtheorem{theorem}{Theorem}[section]
\newtheorem{corollary}[theorem]{Corollary}
\newtheorem{proposition}[theorem]{Proposition}

\newtheorem{lemma}[theorem]{Lemma}
\theoremstyle{definition}

\newtheorem{example}[theorem]{Example}

\theoremstyle{remark}

\usepackage{graphicx}
\usepackage{amsfonts}
\usepackage{amssymb}
\usepackage{amsmath}
\usepackage{color}
\usepackage[inner=1in,outer=1in,bottom=1in,top=1in]{geometry}

\newcommand{\Tr}{\mathrm{Tr}}
\newcommand{\Sym}{\mathrm{Sym}}
\def\XX{{\bf X}}

\def\+{{\oplus}}

\makeatletter
\def\Ddots{\mathinner{\mkern1mu\raise\p@
\vbox{\kern7\p@\hbox{.}}\mkern2mu
\raise4\p@\hbox{.}\mkern2mu\raise7\p@\hbox{.}\mkern1mu}}
\makeatother

\DeclareMathOperator{\lcm}{lcm}

\DeclareMathOperator{\CIRC}{circ}

\begin{document}
\title[Closed formulas for exponential sums]{Closed formulas for exponential sums of symmetric polynomials over Galois fields}

\author{Francis N. Castro}
\address{Department of Mathematics, University of Puerto Rico, 17 Ave. Universidad STE 1701, San Juan, PR 00925}
\email{franciscastr@gmail.com}

\author{Luis A. Medina}
\address{Department of Mathematics, University of Puerto Rico, 17 Ave. Universidad STE 1701, San Juan, PR 00925}
\email{luis.medina17@upr.edu}

\author{L. Brehsner Sep\'ulveda}
\address{Department of Mathematics, University of Puerto Rico, 17 Ave. Universidad STE 1701, San Juan, PR 00925}
\email{leonid.sepulveda1@upr.edu}

\begin{abstract}
Exponential sums have applications to a variety of scientific fields, including, but not limited to, cryptography, coding theory and information theory.  Closed formulas for exponential sums of symmetric Boolean functions were found by Cai, Green and Thierauf in the late 1990's.  Their closed formulas imply that these exponential sums are linear recursive.  The linear recursivity of these sums has been exploited in numerous papers and has been used to compute the asymptotic behavior of such sequences. In this article, we extend the result of Cai, Green and Thierauf, that is, we find closed formulas for exponential sums of symmetric polynomials over any Galois fields.  Our result also implies that the recursive nature of these sequences is not unique to the binary field, as they are also linear recursive over any finite field.  In fact, we provide explicit linear recurrences with integer coefficients for such sequences.  As a byproduct of our results, we discover a link between exponential sums of symmetric polynomials over Galois fields and a problem for multinomial coefficients which similar to the problem of bisecting binomial coefficients.
\end{abstract}

\subjclass[2010]{05E05, 11T23, 11B37}
\date{\today}
\keywords{Exponential sums, symmetric functions, linear recurrences}

\maketitle
\section{Introduction}
Combinatorics and number theory are classic areas of mathematics with fascinating objects that captivate the attention of mathematicians.  One subject that lies in the intersection of these two areas is the theory of Boolean functions.  These beautiful functions have plenty of applications to different scientific fields.  Some examples include electrical engineering, game theory, cryptography, coding theory and information theory.  

An $n$-variable {\em Boolean function} is a function $F({\bf X})$ from the vector space $\mathbb{F}_2^n$ to $\mathbb{F}_2$ where $\mathbb{F}_2=\{0,1\}$ is the binary field and $n$ is a positive number.  In some applications related to cryptography it is important for Boolean functions to be balanced.  A {\it balanced Boolean function} is one for which the number of zeros and the number of ones are equal in its truth table (output table). Balancedness of Boolean functions can be studied from the point of exponential sums.  The {\em exponential sum} of a Boolean function $F({\bf X})$ over $\mathbb{F}_2$ is defined as
\begin{equation}
\label{expS}
S(F)=\sum_{{\bf x}\in \mathbb{F}_2^n} (-1)^{F({\bf x})}.
\end{equation}
Observe that a Boolean function is balanced if and only if $S(F)=0$.

Memory restrictions of current technology have made the problem of efficient implementations of Boolean functions a challenging one. In general, this problem is very hard to tackle, but imposing conditions on these functions may ease the problem. For instance, symmetric Boolean functions are good candidates for efficient implementations and today they are an active area research  \cite{cai, cm1, cm2, cm3, ccms, cusick4, cusick2}.  

In general, to find closed formulas for exponential sums of symmetric Boolean functions was an open problem until Cai, Green and Thierauf found formulas for them in the 1990's \cite{cai}. 
Moreover, their formulas imply that exponential sums of symmetric Boolean functions have a recursive nature.  This has been exploited in \cite{cgm2,cm1,cm2,cm3,cusick4}.   In the particular case of \cite{cm1}, the recursive nature of these sequences and their closed formulas were used to prove asymptotically a conjecture about the balancedness of elementary symmetric Boolean polynomials \cite{cusick2}.

A natural problem to explore is the possibility that these results can be extended to other finite fields or perhaps they are just natural consequences of working over the binary field.
Recently in \cite{ccms}, it has been showed that exponential sums of linear combinations of elementary symmetric polynomials over Galois fields also satisfy linear recurrences.  Therefore, at least the recursive nature of these sequences is 
not unique to the binary field.  

The recursive nature of exponential sums of symmetric polynomials over Galois fields presented in \cite{ccms} did not include explicit linear recurrences for these sequences.  Instead, they proved the existence of such recurrences and provided a method to find them.  In this article, we find explicit linear recurrences for these sequences.  This is done by providing closed formulas for exponential sums of symmetric polynomials over Galois fields. In other words, in this paper we settle the problem of finding closed formulas for exponential sums of linear combinations of elementary symmetric polynomials over any Galois field.  This extends the work of Cai, Green and Thierauf for the binary field \cite{cai} to every finite field.  As far as we know, this is new.  

Our closed formulas depend on some multinomial sum expressions for our exponential sums.  These expressions provide a link between exponential sums of symmetric polynomials over Galois fields and a problem for multinomial 
coefficients which is similar to the problem of bisecting binomial coefficients.  A solution $(\delta_0,\delta_1,\cdots, \delta_n)$ to the equation
\begin{equation}
 \sum_{j=0}^n \delta_j \binom{n}{j}=0,\,\,\, \delta_j \in \{-1,1\},
\end{equation}
is said to give a {\em bisection of the binomial coefficients} $\binom{n}{j}$, $0\leq j \leq n.$  Observe that a solution to (\ref{bisec}) provides us with two disjoints sets $A,B$ such that 
$A\cup B =\{0,1,2,\cdots,n\}$ and
\begin{equation}
 \sum_{j \in A} \binom{n}{j}=\sum_{j\in B}\binom{n}{j}=2^{n-1}.
\end{equation}
The problem of bisecting binomial coefficients is a very interesting problem in its own right, but it is out of the scope of this work.  However, we believe that the connection between exponential
sums of symmetric polynomials and a problem similar to bisecting binomial coefficients is very appealing and underlines the balancedness of symmetric polynomials over finite fields.

This article is divided as follows. The next section contains some preliminaries.  In Section \ref{multinomialsums} we provide multinomial sum expressions for exponential sums of symmetric polynomials
over Galois fields.  We also include some representations that depend on integer partitions.  These multinomial sums representations are a computational improvement over the formal definition of
exponential sums.  Moreover, as just mentioned, they provide a connection to a problem similar to the problem of bisecting binomial coefficients.  Section \ref{closedformulas} is the core and final section 
of this article.  It is also the section where the main results are presented.  In particular, we find closed formulas for some multinomial sum.  This, together with multinomial sum representations
for our exponential sums, allow us to prove closed formulas for exponential sums of symmetric polynomials over finite fields. We also provide explicit linear recurrences  for such exponential sums, showing
that the recursive nature of these sequences is not special to the binary case.  Moreover, every multi-variable function over a finite field extension of $\mathbb{F}_2$ can be identified with a Boolean 
function.  Thus, these results also provide new families of Boolean functions that might be useful for efficient implementations.

\section{Preliminaries}
It is a well-established result in the theory of Boolean functions that any symmetric Boolean function can be identified with a linear combination of elementary symmetric Boolean polynomials.
To be more precise, let $\boldsymbol{e}_{n,k}$ be the elementary symmetric polynomial in $n$ variables of degree $k$. For example,
\begin{equation*}
\boldsymbol{e}_{4,3} = X_1 X_2 X_3\oplus X_1 X_4 X_3\oplus X_2 X_4 X_3\oplus X_1 X_2 X_4,
\end{equation*}
where $\oplus$ represents addition modulo 2.  Every symmetric Boolean function $F({\bf X})$ can be identified with an expression of the form
\begin{equation}
\label{genboolsym}
F(\XX)=\boldsymbol{e}_{n,k_1}\oplus \boldsymbol{e}_{n,k_2}\oplus \cdots\oplus \boldsymbol{e}_{n,k_s},
\end{equation}
where $0\leq k_1<k_2<\cdots<k_s$ are integers.  For the sake of simplicity, the notation $\boldsymbol{e}_{n,[k_1,\ldots,k_s]}$ is used to denote (\ref{genboolsym}).  For example,
\begin{eqnarray}
\boldsymbol{e}_{3,[2,1]}&=&\boldsymbol{e}_{3,2}\oplus \boldsymbol{e}_{3,1}\\ \nonumber
&=& X_1 X_2\oplus X_3 X_2\oplus X_1 X_3\oplus X_1\oplus X_2\oplus X_3.
\end{eqnarray}

As mentioned in the introduction, it is known that exponential sums of symmetric Boolean functions are linear recursive \cite{cai, cm1}.  Moreover, closed formulas for exponential sums of symmetric Boolean 
functions are well known.  In fact, Cai et al. \cite{cai} proved the following theorem.
\begin{theorem}[\cite{cai}]
\label{invariant}
Let $1\leq k_1< \cdots <k_s$ be fixed integers and $r=\lfloor \log_2(k_s) \rfloor+1$. The value of the exponential sum $S(\boldsymbol{e}_{n,[k_1,\cdots,k_s]})$ is given by
\begin{eqnarray}
\label{genvalue}\nonumber
S(\boldsymbol{e}_{n,[k_1,\cdots,k_s]})&=& c_0(k_1,\cdots,k_s) 2^n + \sum_{j=1}^{2^r-1} c_j(k_1,\cdots,k_s) (1+\zeta_j)^n,
\end{eqnarray}
where $\zeta_j = e^{\frac{\pi i\, j}{2^{r-1}}}, i=\sqrt{-1}$ and 
\begin{equation}
\label{coeffs}
c_j(k_1,\cdots,k_s)=\frac{1}{2^r}\sum_{t=0}^{2^r-1}(-1)^{\binom{t}{k_1}+\cdots+\binom{t}{k_s}}\zeta_j^{-t}.
\end{equation}
\end{theorem}
Theorem \ref{invariant} and a closed formula for $c_0(k)$ (proved in \cite{cm1}) were used by Castro and Medina \cite{cm1} to prove asymptotically a conjecture of Cusick, Li and 
St$\check{\mbox{a}}$nic$\check{\mbox{a}}$ about the balancedness of elementary symmetric polynomials \cite{cusick2}.  An adaptation of Theorem \ref{invariant} to perturbations of symmetric Boolean functions
(see \cite{cm2}) was recently used in \cite{cgm2} to prove a generalized conjecture of Canteaut and Videau \cite{canteaut} about the existence of balanced perturbations when the number of variables grows.  The original conjecture, which was stated for symmetric Boolean functions, said that only trivially balanced functions exists when the number of variables grows.  The original conjecture was proved by Guo, Gao and Zhao \cite{ggz}.  The same behavior holds true for perturbations of symmetric Boolean functions.

One of the goals of this article is to generalize Theorem \ref{invariant} to the general setting of Galois fields.  
If $F:\mathbb{F}_q^n\to \mathbb{F}_q$, then its {\it exponential sum over $\mathbb{F}_q$} is given by 
\begin{equation}
\label{expSq}
 S_{\mathbb{F}_q}(F)=\sum_{{\bf x}\in \mathbb{F}^n_q} e^{\frac{2\pi i}{p} \text{Tr}_{\mathbb{F}_q/\mathbb{F}_p}(F({\bf x}))},
\end{equation}
where $\Tr_{\mathbb{F}_q/\mathbb{F}_p}$ represents the field trace function from $\mathbb{F}_q$ to $\mathbb{F}_p$. The {\it field trace function} can be explicitly defined as
\begin{equation}
 \text{Tr}_{\mathbb{F}_{p^l}/\mathbb{F}_p}(\alpha)=\sum _{j=0}^{l-1} \alpha ^{p^j},
\end{equation}
with arithmetic done in $\mathbb{F}_{p^l}$.  Recently in \cite{ccms}, it was proved that exponential sums over $\mathbb{F}_q$ of linear
combinations of elementary symmetric polynomials are linear recurrent with integer coefficients. Thus, the recursive nature of these sequences is not restricted to $\mathbb{F}_2$. The approach presented in \cite{ccms},
however, does not provide specific linear recurrences for these functions.  Instead, it gives a procedure that relies on linear algebra to calculate them.  A closed formula for these sequences, like the one presented in 
Theorem \ref{invariant}, would allow us to find such recurrences.  Perhaps it can also be used to settle, at least asymptotically, the generalization of Cusick, Li and St$\check{\mbox{a}}$nic$\check{\mbox{a}}$
conjecture for Galois fields, see \cite{acgmr}.

The formal definition of an exponential sum is not very useful if one desires to calculate the value of $S_{\mathbb{F}_q}(F)$.  In fact, in general, this problem is clearly exponentially hard.  
However, imposing conditions on the function $F$ sometimes simplifies matters.  For example, in the case of symmetric Boolean functions, it is not hard to show that
\begin{equation}
\label{binomS}
 S(\boldsymbol{e}_{n,[k_1,\cdots,k_s]}) = \sum_{j=0}^n (-1)^{\binom{j}{k_1}+\cdots+\binom{j}{k_s}} \binom{n}{j}.
\end{equation}
Equation (\ref{binomS}) is a clear computational improvement over (\ref{expS}).   It also connects (as mentioned in the introduction) the problem of balancedness of symmetric Boolean functions to the problem of 
bisecting binomial coefficients (see Mitchell \cite{mitchell}).   As mentioned in the introduction, a solution $(\delta_0,\delta_1,\cdots, \delta_n)$ to the equation
\begin{equation}
\label{bisec}
 \sum_{j=0}^n \delta_j \binom{n}{j}=0,\,\,\, \delta_j \in \{-1,1\},
\end{equation}
is said to give a {\em bisection of the binomial coefficients} $\binom{n}{j}$, $0\leq j \leq n.$  
The problem of bisecting binomial coefficients is an interesting problem in its own right, however, it is out of the scope of this work.  The interested reader is invited to read \cite{ims,mitchell}.

In the next section, we proved a formula similar to (\ref{binomS}) for $S_{\mathbb{F}_q}(\boldsymbol{e}_{n,k})$ using multinomial coefficients.  The formula is not only a computational improvement over
the formal definition of $S_{\mathbb{F}_q}(F)$, but also provide a connection to a problem similar to the problem of bisecting of binomial coefficients for multinomial coefficients.  Moreover, the fact that
exponential sums of symmetric polynomials over finite fields can be expressed as multinomial sums is later used in the proof of closed formulas for them.  The proof of the closed formulas also depends on
a classical result in number theory known as Lucas' Theorem.  We decided to include it here for completeness purposes.

\begin{theorem}[Lucas' Theorem]  Suppose that $n$ and $k$ are non-negative integers and let $p$ be a prime.  Suppose that
\begin{eqnarray*}
 n &=& n_0+n_1p+\cdots+n_lp^l\\
 k &=& k_0+k_1p+\cdots+k_lp^l,
\end{eqnarray*}
with $0\leq n_j,k_j <p$ for $j=1,\cdots, l$. Then,
$$\binom{n}{k} \equiv \prod_{j=0}^l\binom{n_j}{k_j}\mod p.$$
\end{theorem}
Let $D=p^{\lfloor\log_p(k)\rfloor+1}$.  Observe that one consequence of Lucas' Theorem is 
\begin{equation}
 \binom{n+D}{k} \equiv \binom{n}{k}\mod p.
\end{equation}
This will be used throughout the rest of the paper.


\section{A formula for exponential sums in terms of multinomial sums}
\label{multinomialsums}

In this section we prove a formula for $S_{\mathbb{F}_q}(\boldsymbol{e}_{n,k})$ in terms of multinomial coefficients.  This formula is a computational improvement over (\ref{expSq}).  We start by finding
a formula, in this case, a recursive one, for the value of $\boldsymbol{e}_{n,k}$ at a vector ${\bf x}$.  

Let $n, k$ and $m$ be positive integers and $a_s$ be a parameter 
($s$ a positive integer).  Let 
\begin{equation}
\Lambda_{a_1}(k,m)=a_1^k \binom{m}{k} 
\end{equation}
and define $\Lambda_{a_1,\cdots, a_l}$ recursively by
\begin{equation}
 \Lambda_{a_1,a_2,\cdots, a_{l+1}}(k,m_1,m_2,\cdots,m_{l+1})=
 \sum_{j=0}^{m_{l+1}}\binom{m_{l+1}}{j}a_{l+1}^j\Lambda_{a_1,\cdots,a_l}(k-j,m_1,m_2,\cdots,m_{l}),
\end{equation}
The value of $\boldsymbol{e}_{n,k}$ is linked to $\Lambda_{a_1,\cdots, a_l}$.

\begin{lemma}
\label{lemmaSymm}
 Let $n$ and $k$ be positive integers.  Let $A_l=\{0,a_1,\cdots,a_l\}$ and ${\bf x} \in A_l^n$.  Suppose that $a_j$ appears $m_j$ times in ${\bf x}$.  Then,
 \begin{equation}
 \boldsymbol{e}_{n,k}({\bf x})=\Lambda_{a_1,\cdots,a_l}(k,m_1,\cdots,m_{l}). 
 \end{equation}
 \end{lemma}
\begin{proof}
First observe that if $l=1$, that is, ${\bf x} \in A_1^n$, then
\begin{equation}
 \boldsymbol{e}_{n,k}({\bf x}) = a_1^k\binom{m_1}{k}.
\end{equation}
Now observe that if the variables $X_n,X_{n-1},\cdots, X_{n-r+1}$ are set to be $\alpha$, then 
 \begin{equation}
  \boldsymbol{e}_{n,k}(X_1,\cdots, X_{n-r},\alpha,\cdots,\alpha)=\sum_{j=0}^r \binom{r}{j}\alpha^j\boldsymbol{e}_{n-r,k-j}(X_1,\cdots, X_{n-r}).
 \end{equation}
Symmetry and an induction argument  finish the proof.
\end{proof}

The above lemma can be used to express exponential sums of symmetric polynomials as a multi-sum of products of multinomial coefficients.

\begin{theorem}
\label{expsumformthm}
Let $n,k$ be natural numbers such that $k\leq n$, $p$ a prime and $q=p^r$ for some positive integer $r$. Suppose that $\mathbb{F}_q=\{0,\alpha_1,\cdots, \alpha_{q-1}\}$ is the Galois field of $q$ elements. 
Then,
\begin{eqnarray*}
S_{\mathbb{F}_q}(\boldsymbol{e}_{n,k})&=&\sum_{m_1=0}^n \sum_{m_2=0}^{n-m_1}\sum_{m_3=0}^{n-m_1-m_{2}}\cdots\sum_{m_{q-1=0}}^{n-m_{1}-\cdots-m_{q-2}} {n\choose m_0^*,m_1,m_2,\cdots, m_{q-1}}\\
&& \times \exp\left(\frac{2\pi i}{p} \Tr_{\mathbb{F}_q/\mathbb{F}_p}(\Lambda_{\alpha_1,\cdots, \alpha_{q-1}}(k,m_1,\cdots, m_{q-1})\right),
\end{eqnarray*}
where $m_{0}^* = n-(m_1+\cdots+m_{q-1})$.
\end{theorem}

\begin{proof}
 Consider a tuple ${\bf x} \in \mathbb{F}_q^n$.  Suppose that $\alpha_j$ appears $m_j$ times in ${\bf x}$.  Clearly,
 this implies
 $$n=m_0^*+m_1+m_2+\cdots+m_{q-1}.$$
 A simple counting argument shows that there are
 \begin{equation}
 \label{multbinom}
  \binom{n}{m_1}\binom{n-m_1}{m_2}\binom{n-m_1-m_2}{m_3}\cdots \binom{n-m_1-m_2-\cdots-m_{q-2}}{m_{q-1}}
 \end{equation}
of such tuples.  This number can be written in multinomial form as
\begin{equation}
 {n\choose m_0^*,m_1,m_2,\cdots, m_{q-1}}.
\end{equation}
Observe that Lemma \ref{lemmaSymm} implies that the value of $\boldsymbol{e}_{n,k}$ on each of these tuples is 
\begin{equation}
 \boldsymbol{e}_{n,k}({\bf x})=\Lambda_{\alpha_1,\cdots, \alpha_{q-1}}(k,m_1,\cdots, m_{q-1}).
\end{equation}
Adding over all possible choices of $m_{1},m_{2},\cdots, m_{q-1}$ produces the result.
\end{proof}
An easy adjustment to the proof of Theorem \ref{expsumformthm} leads the following corollary.

\begin{corollary}
\label{expsumformcoro}
Let $1\leq k_1<k_2<\cdots<k_s$ and $n$ be positive integers, $p$ a prime and $q=p^r$ for some positive integer $r$. Suppose that $\mathbb{F}_q=\{0,\alpha_1,\cdots, \alpha_{q-1}\}$ is the Galois field of $q$ elements. 
Consider the symmetric function
$$\sum_{j=1}^s \beta_j \boldsymbol{e}_{n,k_j}\,\,\, \text{ where }\beta_j\in \mathbb{F}_q^{\times}.$$
Then,
\begin{eqnarray*}
S_{\mathbb{F}_q}\left(\sum_{j=1}^s \beta_j\boldsymbol{e}_{n,k_j}\right)&=&\sum_{m_1=0}^n \sum_{m_2=0}^{n-m_1}\sum_{m_3=0}^{n-m_1-m_{2}}\cdots\sum_{m_{q-1=0}}^{n-m_{1}-\cdots-m_{q-2}} {n\choose m_0^*,m_1,m_2,\cdots, m_{q-1}}\\
&& \times \exp\left(\frac{2\pi i}{p} \Tr_{\mathbb{F}_q/\mathbb{F}_p}\left(\sum_{j=1}^s \beta_j\Lambda_{\alpha_1,\cdots, \alpha_{q-1}}(k_j,m_1,\cdots, m_{q-1})\right)\right).
\end{eqnarray*}
\end{corollary}

\begin{proof}
 The proof follows the same argument as in Theorem \ref{expsumformthm}.
\end{proof}

Theorem \ref{expsumformthm} and its corollary can be written in terms of partitions of $n$.  We say that $\boldsymbol{\lambda}= (\lambda_1,\cdots,\lambda_r)$ is a {\em partition} of $n$, and write $\boldsymbol{\lambda} \dashv n$, if the $\lambda_j$ are integers and
$$\lambda_1\geq \cdots \geq \lambda_r\geq 1 \,\, \text{ and }\,\,n=\lambda_1+\cdots+\lambda_r.$$
The notation $\boldsymbol{\lambda} \dashv_q n$ implies that $\boldsymbol{\lambda}$ is a partition of $n$ and has at most $q$ entries.  For example, if $\boldsymbol{\lambda}=(6,3,1)$, then
$\boldsymbol{\lambda} \dashv_4 10$ because it has 3 entries and $3\leq 4$.  On the other hand, if $\boldsymbol{\lambda} =(4,2,2,1,1)$, then  $\boldsymbol{\lambda} \dashv 10$, but $\boldsymbol{\lambda} \not \dashv_4 10$.  From now on, we will see partitions $\boldsymbol{\lambda} \dashv_q n$ as lists of length $q$.  Of course, by definition, a partition $\boldsymbol{\lambda} \dashv_q n$ may have less than $q$ entries.  If that is the case, right-pad zeros to the list until it has $q$ entries.  For example, $\boldsymbol{\lambda}=(6,3,1)$ is such that $\boldsymbol{\lambda} \dashv_4 10$.  In this case, we view $\boldsymbol{\lambda}$ as $\boldsymbol{\lambda} =(6,3,1,0)$.

If $\boldsymbol{\lambda}  \dashv n$, then the symbol
\begin{equation*}
\binom{n}{\boldsymbol{\lambda}}
\end{equation*}
represents the multinomial obtained from $\boldsymbol{\lambda}$. For example, if $\boldsymbol{\lambda} =(6,3,1)$, then
$$\binom{10}{\boldsymbol{\lambda}}={10 \choose 6,3,1}.$$
By a {\em rearrangement} of  $\boldsymbol{\lambda}$ we mean a permutation of the symbols in $\boldsymbol{\lambda}$.  For example, the set of all different rearrangements of  $\boldsymbol{\lambda} =(2,2,1,1)$ is
\begin{eqnarray*}
(2, 2, 1, 1), & (2, 1, 2, 1)\\
(2, 1, 1, 2), & (1, 2, 2, 1)\\
(1, 2, 1, 2), & (1, 1, 2, 2).
\end{eqnarray*}
We use $\Sym(\boldsymbol{\lambda})$ to denote the set of all rearrangements of $\boldsymbol{\lambda}$.  Finally, if $\boldsymbol{\gamma}$ is a non-empty list, then $\boldsymbol{\gamma}^*$ is the list obtained 
from $\boldsymbol{\gamma}$ by removing the first element.  For example, if $\boldsymbol{\gamma}=(2,2,1,1)$, then $\boldsymbol{\gamma}^*=(2,1,1)$.  Theorem \ref{expsumformthm} 
and Corollary \ref{expsumformcoro} can be re-stated as follows.
\begin{theorem}
\label{thmpart}
Let $n,k$ be natural numbers such that $k\leq n$, $p$ a prime and $q=p^r$ for some positive integer $r$. Suppose that $\mathbb{F}_q=\{0,\alpha_1,\cdots, \alpha_{q-1}\}$ is the Galois field of $q$ elements. 
Then,
\begin{eqnarray*}
S_{\mathbb{F}_q}(\boldsymbol{e}_{n,k})&=&\sum_{\boldsymbol{\lambda}\dashv_q n} \binom{n}{\boldsymbol{\lambda}} \sum_{\boldsymbol{\gamma}\in \Sym(\boldsymbol{\lambda})}
 \exp\left(\frac{2\pi i}{p} \Tr_{\mathbb{F}_q/\mathbb{F}_p}(\Lambda_{\alpha_1,\cdots, \alpha_{q-1}}(k,\boldsymbol{\gamma}^*)\right).
\end{eqnarray*}
\end{theorem}

\begin{corollary}
\label{coropart}
Let $1\leq k_1<k_2<\cdots<k_s$ and $n$ be positive integers, $p$ a prime and $q=p^r$ for some positive integer $r$. Suppose that $\mathbb{F}_q=\{0,\alpha_1,\cdots, \alpha_{q-1}\}$ is the Galois field of $q$ elements. 
Consider the symmetric function
$$\sum_{j=1}^s \beta_j \boldsymbol{e}_{n,k_j}\,\,\, \text{ where }\beta_j\in \mathbb{F}_q^{\times}.$$
Then,
\begin{eqnarray*}
S_{\mathbb{F}_q}\left(\sum_{j=1}^s \beta_j\boldsymbol{e}_{n,k_j}\right)&=&\sum_{\boldsymbol{\lambda}\dashv_q n} \binom{n}{\boldsymbol{\lambda}} \sum_{\boldsymbol{\gamma}\in \Sym(\boldsymbol{\lambda})}
 \exp\left(\frac{2\pi i}{p} \Tr_{\mathbb{F}_q/\mathbb{F}_p}\left(\sum_{j=1}^s \beta_j\Lambda_{\alpha_1,\cdots, \alpha_{q-1}}(k_j,\boldsymbol{\gamma}^*)\right)\right).
\end{eqnarray*}
\end{corollary}

For small $q$, Theorem \ref{expsumformthm} and the recursive nature of $\Lambda_{a_1,\cdots,a_l}$ can be used to speed up the computation of $S_{\mathbb{F}_q}(\boldsymbol{e}_{n,k})$.  Theorem 
\ref{expsumformthm} and Corollary \ref{expsumformcoro} also offers a hint to a problem similar to bisections of binomial coefficients for multinomial coefficients. Emulating the binary case, we define 
$(p,q)$-{\it section} of multinomial coefficients ($q$ being a power of $p$) to be the process of dividing the list
\begin{equation}
 \mathcal{L}(n;q)=\left\{{n\choose m_0,m_1,m_2,\cdots, m_{q-1}^*}\right\},
\end{equation}
where the indices run
$$0\leq m_0 \leq n, 0\leq m_1\leq n-m_0,\cdots,0\leq m_{q-2}\leq n-m_0-m_1-\cdots-m_{q-3},$$
into $p$ sublists, $l_j(n;q), 1\leq j \leq p$, such that the sum on each sublist is the same. This common sum must be $q^{n-1}$. Observe that every time 
$S_{\mathbb{F}_q}(\beta_1\boldsymbol{e}_{n,k_1}+\cdots+\beta_s\boldsymbol{e}_{n,k_s})=0$ we obtain
a $(p,q)$-section to of multinomial coefficients.  This connection generalizes the one that exists between bisections of binomial coefficients and symmetric Boolean functions.
\begin{example}
The elementary symmetric polynomial $\boldsymbol{e}_{5,3}$ is such that $S_{\mathbb{F}_3}(\boldsymbol{e}_{5,3})=0$.  
Observe that
\begin{equation}
 \mathcal{L}(5;3) =\{1, 5, 10, 10, 5, 1, 5, 20, 30, 20, 5, 10, 30, 30, 10, 10, 20, 10, 5, 5, 1\}.
\end{equation}
The 3-section that corresponds to
$\boldsymbol{e}_{5,3}$ over $\mathbb{F}_3$ is
\begin{eqnarray}
 l_1(5;3) &=& \{1, 5, 5, 10, 10, 20, 30\}\\\nonumber
 l_2(5;3) &=& \{1, 5, 5, 10, 10, 20, 30\}\\\nonumber
 l_3(5;3) &=& \{1, 5, 5, 10, 10, 20, 30\}.
\end{eqnarray} 
\end{example}

\begin{example}
The symmetric polynomial $\boldsymbol{e}_{6,5}+\boldsymbol{e}_{6,3}$ also satisfies $S_{\mathbb{F}_3}(\boldsymbol{e}_{6,5}+\boldsymbol{e}_{6,3})=0$.  
In this case,
\begin{equation}
 \mathcal{L}(6;3) =\{1, 6, 15, 20, 15, 6, 1, 6, 30, 60, 60, 30, 6, 15, 60, 90, 60, 15, 20, 60, 60, 20, 15, 30, 15, 6, 6, 1\}.
\end{equation}
The 3-section that corresponds to $\boldsymbol{e}_{6,5}+\boldsymbol{e}_{6,3}$ over $\mathbb{F}_3$ is
\begin{eqnarray}
\label{3sec}
 l_1(6;3) &=& \{1, 6, 6, 15, 15, 20, 30, 30, 30, 90\}\\\nonumber
 l_2(6;3) &=& \{1, 6, 6, 15, 15, 20, 60, 60, 60\}\\\nonumber
 l_3(6;3) &=& \{1, 6, 6, 15, 15, 20, 60, 60, 60\}.
\end{eqnarray}  
\end{example}

As in the Boolean case, we may try to define trivial $(p,q)$-sections.  A possible way to do this is to say that a $(p,q)$-section is trivial if $l_1(n;k)=l_2(n;k)=\cdots=l_{p}(n;k)$.   Again, following the
binary case, we say that a symmetric polynomial $\beta_1\boldsymbol{e}_{n,k_1}+\cdots+\beta_s\boldsymbol{e}_{n,k_s}$ is trivially balanced over $\mathbb{F}_q$ if its related $(p,q)$-section is trivial.
For example, $\boldsymbol{e}_{5,3}$ is trivially balanced, while $\boldsymbol{e}_{6,5}+\boldsymbol{e}_{6,3}$ is not.  It would be interesting to know if some results known for the binary case also apply to this problem.  

Exponential sums of linear combinations of elementary symmetric polynomials are also linked, via Theorem \ref{thmpart} and Corollary \ref{coropart}, to the Diophantine equation
\begin{equation}
\label{diopheq}
 \sum_{\boldsymbol{\lambda}\dashv_q n} \binom{n}{\boldsymbol{\lambda}} x_{\boldsymbol{\lambda}}=0.
\end{equation}
Observe that every time 
$$S_{\mathbb{F}_q}\left(\sum_{j=1}^s \beta_j \boldsymbol{e}_{n,k_j}\right)=0,$$
we find a solution to (\ref{diopheq}).

\begin{example}
 Consider $\mathbb{F}_4=\{0,1,\alpha,\alpha+1\}$ where $\alpha^2=\alpha+1$.  The symmetric polynomial 
 $$(1+\alpha)\boldsymbol{e}_{n,3}+(1+\alpha)\boldsymbol{e}_{n,2}+\alpha\boldsymbol{e}_{n,1}$$
 is such that
\begin{equation}
\label{diopheqex}
S_{\mathbb{F}_4}\left((1+\alpha)\boldsymbol{e}_{8,3}+(1+\alpha)\boldsymbol{e}_{8,2}+\alpha\boldsymbol{e}_{8,1}\right)=0.
\end{equation}
Therefore, we have a solution to (\ref{diopheq}) for $n=8$ and $q=4$.  The integer partitions $\boldsymbol{\lambda}$  of $8$ that satisfies $\boldsymbol{\lambda}\dashv_4 8$ are
$$ \begin{array}{llll}
 \boldsymbol{\lambda}_1=(8), &  \boldsymbol{\lambda}_2=(7, 1), & \boldsymbol{\lambda}_3= (6, 2), & \boldsymbol{\lambda}_4=(6, 1, 1),\\
 \boldsymbol{\lambda}_5=(5, 3),& \boldsymbol{\lambda}_6=(5, 2, 1), & \boldsymbol{\lambda}_7=(5, 1, 1, 1),& \boldsymbol{\lambda}_8= (4, 4),\\
 \boldsymbol{\lambda}_9=(4, 3, 1),& \boldsymbol{\lambda}_{10}= (4, 2, 2),& \boldsymbol{\lambda}_{11}= (4, 2, 1, 1),& \boldsymbol{\lambda}_{12}= (3, 3, 2),\\
 \boldsymbol{\lambda}_{13}=(3, 3, 1, 1),& \boldsymbol{\lambda}_{14}=(3,2, 2, 1),& \boldsymbol{\lambda}_{15}=(2, 2, 2, 2). & 
 \end{array}
$$
The solution to (\ref{diopheq}) provided by (\ref{diopheqex}) is given by
$$
(\delta_1,\delta_2,\cdots,\delta_{15})=(4, -4, -4, 4, -4, 8, -4, 6, -8, -4, 4, 4, 2, -4, 1).
$$
In other words,
$$\sum_{j=1}^{15} \binom{8}{\boldsymbol{\lambda}_j} \delta_j=0.$$
\end{example}
A natural problem to explore is to see how solutions to (\ref{diopheq}) given by exponential sums of linear combinations of elementary symmetric polynomials look like as $n$ grows.  Perhaps something similar to 
the study presented in \cite{cgm2} holds true in this case.  This is part of future research.

In the next section, we prove closed formulas for exponential sums of symmetric polynomials
over Galois fields.  Moreover, we provide explicit linear recurrences with integer coefficients for these exponential sums.


\section{Closed formulas for exponential sums of symmetric polynomials}
\label{closedformulas}
In this section we generalize Theorem \ref{invariant}, that is, we provide closed formulas for the exponential sums considered in this article.  These formulas, in turn, allow us to find explicit recursions for these sequences.  Our formulas depend on circulant matrices and on periodicity.  Thus, we start with a short background on these topics.

Let $D$ be a positive integer and $\alpha=\left(c_{0},c_{1},\ldots,c_{D-1}\right)\in\mathbb{C}^{D}$.  The $D$-{\em circulant matrix} associated to $\alpha$, denoted by $\text{circ}(\alpha)$, is defined by

\begin{equation}
\text{circ}(\alpha):=\left(\begin{array}{ccccc}
c_{0} & c_{1} & \cdots & c_{D-2} & c_{D-1}\\
c_{D-1} & c_{0} & \cdots & c_{D-1} & c_{D-2}\\
\vdots & \vdots & \ddots & \vdots & \vdots\\
c_{2} & c_{3} & \cdots & c_{0} & c_{1}\\
c_{1} & c_{2} & \cdots & c_{D-1} & c_{0}
\end{array}\right).
\end{equation}
The polynomial $p_{\alpha}(X)=c_0+c_1 X+\cdots+ c_{D-1} X^{D-1}$ is called the {\em associated polynomial} of the circulant matrix.  In the literature, this polynomial is also called {\em representer polynomial}.
Observe that if
\begin{equation}
 \pi = \left(
\begin{array}{rrrrrr}
 0 & 1 & 0 & \cdots  & 0 & 0 \\
 0 & 0 & 1 & \cdots  & 0 & 0 \\
 \vdots & \vdots & \vdots & \ddots & \vdots & \vdots\\
 0 & 0 & 0 &  \cdots & 0 & 1 \\
 1 & 0 & 0 & \cdots & 0 & 0 \\
\end{array}
\right),
\end{equation}
then $\text{circ}(\alpha) = p_{\alpha}(\pi)$.

Circulant matrices are well-understood objects.  For example, it is known that the (normalized) eigenvectors of any circulant matrix $\text{circ}(\alpha)$ are given by
\begin{equation}
v_j=\frac{1}{\sqrt{n}}(1,\omega_j,\omega_j^2,\cdots, \omega_j^{D-1})^T,
\end{equation}
where $\omega_j=\exp\left(2\pi \mathrm{i} j/D\right)$ and $i=\sqrt{-1}$, with corresponding eigenvalues
\begin{equation}
\lambda_j(\alpha) = p_{\alpha}(\omega_j)=c_0+c_1 \omega_j+c_2 \omega_j^2+\cdots+c_{D-1} \omega_j^{D-1}.
\end{equation}
Moreover, any circulant matrix $\text{circ}(\alpha)$ can be diagonalized in the following form.  Consider the {\it Discrete Fourier Transform} matrix
\begin{equation}
 F_n = \left(
\begin{array}{rrrr}
 \xi_n^{0\cdot 0} & \xi_n^{0\cdot 1} & \cdots &  \xi_n^{0\cdot (n-1)} \\
 \xi_n^{1\cdot 0} & \xi_n^{1\cdot 1} & \cdots &  \xi_n^{1\cdot (n-1)} \\
 \vdots & \vdots &  \ddots & \vdots \\
 \xi_n^{(n-1)\cdot 0} & \xi_n^{(n-1)\cdot 1} & \cdots  & \xi_n^{(n-1)\cdot (n-1)} \\
\end{array}
\right),
\end{equation}
where $\xi_n=\exp(-2\pi i/n)$.  Let $U_n=(1/\sqrt{n})F_n$ be its normalization and define
\begin{equation}
 \Delta(\alpha) = \text{diag}(\lambda_0(\alpha),\lambda_1(\alpha),\cdots,\lambda_{D-1}(\alpha)).
\end{equation}
Then,
\begin{equation}
\label{diagon}
 \text{circ}(\alpha) = U_D \Delta(\alpha) U_D^*.
\end{equation}
See \cite[Th.3.2.2, p. 72]{Davis} for more information.

We say that a function $f:\mathbb{Z}\rightarrow\mathbb{Z}$ is {\em periodic} with period $D$ if $f(j+D)=f(j)$
for any $j\in\mathbb{Z}$.  Periodicity can be extended to functions $g:\mathbb{Z}\times \mathbb{Z}\to\mathbb{Z}$ without too much effort.  The periodicity of a function  $g:\mathbb{Z}\times \mathbb{Z}\rightarrow\mathbb{Z}$ is usually divided by components.  We say that a positive integer $D_1$ is a {\it period in the first component of} $g$ if 
\begin{equation}
g(j_1+D_1,j_2)=g(j_1,j_2) 
\end{equation}
for every $j_1,j_2 \in \mathbb{Z}$.  Similarly, we say that a positive integer $D_2$ is a {\it period in the second component of} $g$ if 
\begin{equation}
g(j_1,j_2+D_2)=g(j_1,j_2)
\end{equation} 
for every $j_1,j_2 \in \mathbb{Z}$. Of course, if $g$ is periodic in its first and second components, then we say that $g$ is periodic.  Moreover, $D=\lcm(D_1,D_2)$ is such that
\begin{equation}
g(j_1+D,j_2+D)=g(j_1,j_2)
\end{equation}
for every $j_1,j_2 \in \mathbb{Z}$.  The concept of periodicity can be extended further to functions from $\mathbb{Z}\times\mathbb{Z}\times\cdots\times \mathbb{Z}$ to $\mathbb{Z}$.  The discussion is the same as for the case $\mathbb{Z}\times \mathbb{Z}$, so we do not write the details.

We are now ready to start with the argument for our formulas. 
Consider the summation
\begin{equation}
\label{sumofint}
\sum_{l=0}^n a^l\binom{n}{l}.
\end{equation}
Later it will become clear why we choose this sum.  Given a positive integer $D>1$, the sum (\ref{sumofint}) can be splitted as 
\begin{equation}
\sum_{l=0}^n a^l\binom{n}{l} = \sum_{t=0}^{D-1} r_t(n;a),
\end{equation}
where 
\begin{equation}
r_t(n;a)=\sum_{j\equiv t \,\,\text{mod }D} a^j\binom{n}{j}.
\end{equation}

\begin{proposition}
\label{closedformprop}
Let $n\in \mathbb{N}$ and $0\leq t \leq D-1$.  Then,
\begin{equation}
r_t(n;a) = \frac{1}{D}\sum_{m=0}^{D-1} \xi_{D}^{t m} \lambda_{m}^n,
\end{equation}
where $\xi_D=\exp(2\pi i/D)$ and $\lambda_m = 1+a\xi_D^{-m}$ are the eigenvalues of $\CIRC(1,0,\cdots, 0,a)$.
\end{proposition}

\begin{proof}
The approach of this proof is similar to the one presented in \cite{cai}.  Note that for $1\leq t \leq D-1$, we have
\begin{equation}
r_t(n;a) = r_t(n-1;a)+a\,r_{t-1}(n-1;a).
\end{equation}
Also,
\begin{equation}
r_0(n;a) = r_0(n-1;a)+a\,r_{D-1}(n-1;a).
\end{equation}
Therefore, if we define
\begin{equation}
 \boldsymbol{r}(n;a)=\left(
\begin{array}{c}
 r_0(n;a) \\
 r_1(n;a) \\
 \vdots \\
 r_{D-1}(n;a) \\
\end{array}
\right),
\end{equation}
then
\begin{eqnarray}
 \boldsymbol{r}(n;a)&=&\left(
\begin{array}{cccccc}
 1 & 0 & 0 & \cdots & 0 & a \\
 a & 1 & 0 & \cdots & 0 & 0 \\
 0 & a & 1 & \cdots & 0 & 0 \\
 \vdots & \vdots & \vdots & \ddots & \vdots & \vdots\\
 0 & 0 & 0 & \cdots & a & 1 \\
\end{array}
\right) \boldsymbol{r}(n-1;a).
\end{eqnarray}
Let $\alpha=(1,0,\cdots, 0, a)$.  The last equation is equivalent to
\begin{equation}
\label{matrec}
 \boldsymbol{r}(n;a)= A_D(a)\boldsymbol{r}(n-1;a),
\end{equation}
where $A_D(a)=\CIRC(\alpha)$.

Iteration of (\ref{matrec}) leads to $\boldsymbol{r}(n;a)= A_D(a)^n\boldsymbol{r}(0;a)$.  Observe that
\begin{equation}
 r_0(0;a)=\binom{0}{0} = 1 \,\, \text{ and }\,\, r_t(0;a)=0 \text{ for }t>0.
\end{equation}
Thus,
\begin{equation}
\boldsymbol{r}(0;a)=\left(\begin{array}{c}
 1 \\
 0 \\
 \vdots \\
 0 \\
\end{array}
\right).
\end{equation}
Equation (\ref{diagon}) now implies that 
\begin{eqnarray}
 \boldsymbol{r}(n;a) &=& \frac{1}{D} U_D \Delta(\alpha)^n U_D^* \boldsymbol{r}(0;a)=\frac{1}{D} U_D \Delta(\alpha)^n\left(\begin{array}{c}
 1 \\
 1 \\
 \vdots \\
 1 \\
\end{array}
\right)\\\nonumber
&=& \frac{1}{D} U_D \left(\begin{array}{c}
 \lambda_0(\alpha)^n \\
 \lambda_1(\alpha)^n \\
 \vdots \\
 \lambda_{D-1}(\alpha)^n \\
\end{array}
\right)= \left(\begin{array}{c}
 \frac{1}{D}\sum_{j=0}^{D-1}\xi_D^{(s-1)j}\lambda_0(\alpha)^n\\
 \frac{1}{D}\sum_{j=0}^{D-1}\xi_D^{(s-1)j}\lambda_1(\alpha)^n\\
 \vdots \\
 \frac{1}{D}\sum_{j=0}^{D-1}\xi_D^{(s-1)j}\lambda_{D-1}(\alpha)^n\\
\end{array}
\right).
\end{eqnarray}
It follows that 
\begin{equation}
 r_t(n;a)=\frac{1}{D}\sum_{j=0}^{D-1}\xi_D^{tj}\lambda_{j}(\alpha)^n
\end{equation}
where $\lambda_j(\alpha) = 1+a\xi_D^{-j}$.
\end{proof}

The following results are easy consequences of the above proposition.

\begin{corollary}
\label{corocloseda}
Let $F$ be a periodic function with period $D$.  Suppose that $\xi^D=1$ (not necessarily primitive).  Then,
 \begin{equation}
  \sum_{l=0}^n \binom{n}{l}a^l\xi^{F(l)} =\frac{1}{D}\sum_{t=0}^{D-1}\xi^{F(t)}\sum_{j=0}^{D-1}\xi_D^{tj}\lambda_j^n,
 \end{equation}
where $\xi_D=\exp(2\pi i/D)$ and $\lambda_j=1+a\xi_D^{-j}$, for $0\leq j \leq D-1$, are the eigenvalues of $\CIRC(1,0,\cdots, 0,a)$.
\end{corollary}

\begin{proof}
 Observe that
 \begin{eqnarray}
  \sum_{l=0}^n \binom{n}{l}a^l\xi^{F(l)}&=& \sum_{t=0}^{D-1}\left(\sum_{j\equiv t \text{ mod }D}\xi^{F(t)}a^l\binom{n}{j}\right)\\\nonumber
  &=& \sum_{t=0}^{D-1}\xi^{F(t)}r_t(n;a).
 \end{eqnarray}
The result now follows from Proposition \ref{closedformprop}.
\end{proof}

\begin{corollary}
\label{coroclosed}
 Let $F$ be a periodic function with period $D$.  Suppose that $\xi^D=1$ (not necessarily primitive).  Then,
 \begin{equation}
  \sum_{l=0}^n\binom{n}{l}\xi^{F(l)} = \frac{1}{D}\sum_{t=0}^{D-1}\xi^{F(t)}\sum_{j=0}^{D-1} \xi_D^{tj}\left(1+\xi_D^{-j}\right)^n,
 \end{equation}
 where  $\xi_D=\exp(2\pi i/D)$.
\end{corollary}

\begin{proof}
 Set $a=1$ in the previous corollary.
\end{proof}

These results can be extended further to obtain closed formulas for multinomial sums.

\begin{theorem}
\label{generalclosed}
Let $F(q_1,\cdots,q_r)$ be a periodic function in each component.  Moreover, suppose that $D$ is a period for $F$ in each component and that $\xi^D=1$ (not necessarily primitive).  Define,
\begin{equation}
\label{gensum}
S(n)=\sum_{q_1=0}^n\sum_{q_2=0}^{n-q_1}\cdots\sum_{q_r=0}^{n-q_1-\cdots-q_{r-1}} \binom{n}{q_1}\binom{n-q_1}{q_2}\cdots\binom{n-q_1-\cdots -q_{r-1}}{q_r}\xi^{F(q_1,\cdots,q_r)}. 
\end{equation}
Then,
\begin{equation}
\label{nicerep}
 S(n)=\frac{1}{D^r}\sum_{b_r=0}^{D-1}\sum_{b_{r-1}=0}^{D-1}\cdots \sum_{b_1=0}^{D-1}\sum_{j_1=0}^{D-1}\sum_{j_2=0}^{D-1}\cdots \sum_{j_r=0}^{D-1} \xi^{F(b_1,\cdots,b_r)}\xi_D^{j_1b_r+\cdots+j_rb_1}
 \lambda_{j_1,\cdots,j_r}^n,
\end{equation}
where  $\xi_D=\exp(2\pi i/D)$ and $\lambda_{j_1,\cdots, j_r} = 1+\xi_D^{-j_1}+\xi_D^{-j_2}+\cdots+\xi_D^{-j_r}$.
\end{theorem}

\begin{proof}
 We present the proof for the case when $r=3$.  We decided to do this in order to simplify the writing of the proof.  The general case is the same argument repeated multiple times.
 
 Write $S(n)$ as
 \begin{equation}
  S(n)=\sum_{q_1=0}^n\sum_{q_2=0}^{n-q_1} \binom{n}{q_1}\binom{n-q_1}{q_2}\sum_{q_3=0}^{n-q_1-q_2}\binom{n-q_{1}-q_2}{q_3}\xi^{F(q_1,q_2,q_3)}.
 \end{equation}
Apply Corollary \ref{coroclosed} to the last sum to get
\begin{equation}
  S(n)=\sum_{q_1=0}^n\sum_{q_2=0}^{n-q_1} \binom{n}{q_1}\binom{n-q_1}{q_2}\left(\frac{1}{D}\sum_{b_3=0}^{D-1}\xi^{F(q_1,q_2,b_3)}
  \sum_{j_1=0}^{D-1}\xi_D^{j_1b_3}\lambda_{j_1}^{n-q_1-q_2}\right),
\end{equation}
where $\lambda_{j_1} = 1+\xi_D^{-j_1}$.  Re-write this equation as
\begin{equation}
\label{after1step}
 S(n)=\frac{1}{D}\sum_{b_3=0}^{D-1}\sum_{j_1=0}^{D-1}\xi_D^{j_1b_3}\sum_{q_1=0}^n\binom{n}{q_1}\lambda_{j_1}^{n-q_1}\sum_{q_2=0}^{n-q_1} \binom{n-q_1}{q_2}(\lambda_{j_1}^{-1})^{q_2}\xi^{F(q_1,q_{2},b_3)}.
\end{equation}
Now apply Corollary \ref{corocloseda} to the last sum to get
\begin{equation}
\label{axuliary1}
 S(n)=\frac{1}{D}\sum_{b_3=0}^{D-1}\sum_{j_1=0}^{D-1}\xi_D^{j_1b_3}\sum_{q_1=0}^n\binom{n}{q_1}\lambda_{j_1}^{n-q_1}\left(\frac{1}{D}\sum_{b_2=0}^{D-1}\xi^{F(q_1,b_2,b_3)}\sum_{j_2=0}^{D-1}\xi_D^{j_2b_2}\right)
 (1+\lambda_{j_1}^{-1}\xi_D^{-j_2})^{n-q_1}.
 \end{equation}
However, observe that 
$$\lambda_{j_1}^{n-q_1}(1+\lambda_{j_1}^{-1}\xi_D^{-j_2})^{n-q_1}=(\lambda_{j_1}+\xi_D^{-j_2})^{n-q_1}=(1+\xi_D^{-j_1}+\xi_D^{-j_2})^{n-q_1}=\lambda_{j_1,j_2}^{n-q_1}.$$
Therefore,
\begin{equation}
\label{axuliary2}
 S(n)=\frac{1}{D}\sum_{b_3=0}^{D-1}\sum_{j_1=0}^{D-1}\xi_D^{j_1b_3}\sum_{q_1=0}^n\binom{n}{q_1}\left(\frac{1}{D}\sum_{b_2=0}^{D-1}\xi^{F(q_1,b_2,b_3)}\sum_{j_2=0}^{D-1}\xi_D^{j_2b_2}\right)
 \lambda_{j_1,j_2}^{n-q_1}.
 \end{equation}
Rearrange terms to get
\begin{align}
 S(n)=\frac{1}{D^2}\sum_{b_3=0}^{D-1}\sum_{b_2=0}^{D-1}\sum_{j_1=0}^{D-1}\sum_{j_2=0}^{D-1}\xi_D^{j_1b_3+j_2b_2}\lambda_{j_1,j_2}^{n}.\sum_{q_1=0}^n\binom{n}{q_1}\xi^{F(q_1,b_2,b_3)}\xi_D^{j_2b_2}(\lambda_{j_1,j_2}^{-1})^{q_1}.
\end{align}
Apply Corollary \ref{corocloseda} once again. After simplification, we have
\begin{align}
 S(n)=\frac{1}{D^3}\sum_{b_3=0}^{D-1}\sum_{b_2=0}^{D-1}\sum_{b_1=0}^{D-1}\sum_{j_1=0}^{D-1}\sum_{j_2=0}^{D-1}\sum_{j_3=0}^{D-1}\xi_D^{j_1b_3+j_2b_2+j_3b_1}\xi^{F(b_1,b_2,b_3)}\lambda_{j_1,j_2,j_3}^n.
\end{align}
The general case follows using the same method.  This concludes the proof.
\end{proof}

Observe that equation (\ref{nicerep}) can be written as 
\begin{equation}
 S(n)=\sum_{j_1=0}^{D-1}\sum_{j_2=0}^{D-1}\cdots \sum_{j_r=0}^{D-1}
 d_{j_1,\cdots, j_r}(D)\lambda_{j_1,\cdots,j_r}^n,
\end{equation}
where
\begin{equation}
 d_{j_1,\cdots, j_r}(D)=\frac{1}{D^r}\sum_{b_r=0}^{D-1}\sum_{b_{r-1}=0}^{D-1}\cdots \sum_{b_1=0}^{D-1}  \xi^{F(b_1,\cdots,b_r)}\xi_D^{j_1b_r+\cdots+j_rb_1}.
\end{equation}
However, note that $\lambda_{t_1,\cdots, t_r}=\lambda_{t'_1,\cdots, t'_r}$ where $(t'_1,\cdots, t'_r)$ is any rearrangement of $(t_1,\cdots, t_r)$.  This means that the coefficient of 
$\lambda^n_{t_1,\cdots,t_r}$ in (\ref{nicerep}) is the sum of all $d_{t_1',\cdots,t_r'}(D)$ where $(t_1',\cdots,t_r')$ is a rearrangement of $(t_1,\cdots,t_r)$.
Recall that $\Sym(t_1,\cdots,t_r)$ represents the set 
of all rearrangements of $(t_1,\cdots, t_r)$.  Theorem \ref{generalclosed} now can be re-stated as follows.

\begin{theorem}
\label{generalclosedA}
Let $F(q_1,\cdots,q_r)$ be a periodic function in each component.  Moreover, suppose that $D$ is a period for $F$ in each component and that $\xi^D=1$ (not necessarily primitive).  Define,
\begin{equation}
\label{gensumA}
S(n)=\sum_{q_1=0}^n\sum_{q_2=0}^{n-q_1}\cdots\sum_{q_r=0}^{n-q_1-\cdots-q_{r-1}} \binom{n}{q_1}\binom{n-q_1}{q_2}\cdots\binom{n-q_1-\cdots -q_{r-1}}{q_r}\xi^{F(q_1,\cdots,q_r)}. 
\end{equation}
Then,
\begin{equation}
 S(n)=\sum_{j_1=0}^{D-1}\sum_{j_2=0}^{j_1}\cdots \sum_{j_r=0}^{j_{r-1}}c_{j_1,\cdots,j_r}(D)\left(1+\xi_D^{-j_1}+\cdots+\xi_D^{-j_r}\right)^n,
\end{equation}
where  
\begin{equation}
 c_{j_1,\cdots, j_r}(D)=\frac{1}{D^r}\sum_{b_r=0}^{D-1}\sum_{b_{r-1}=0}^{D-1}\cdots \sum_{b_1=0}^{D-1}  \xi^{F(b_1,\cdots,b_r)}\sum_{(j_1',\cdots,j_r')\in \Sym(j_1,\cdots, j_r)}
 \xi_D^{j_1'b_r+\cdots+j_r' b_1},
\end{equation}
and $\xi_D=\exp(2\pi i/D)$.
\end{theorem}
\begin{proof}
 This is just a re-statement of Theorem \ref{generalclosed}.
\end{proof}

A nice consequence of this result is that sequences of the form $\{S(n)\}$, with $S(n)$ defined as in (\ref{gensumA}), satisfy linear recurrences with integer coefficients. Moreover, we can 
provide explicit characteristic polynomials for such recurrences.

\begin{corollary}
\label{gencoro}
 Let $S(n)$ be defined as in (\ref{gensumA}).  Then, the sequence $\{S(n)\}$ satisfies the linear recurrence with integer coefficients whose characteristic polynomial is given by
 \begin{equation}
  P_S(X)=\prod_{a_1=0}^{D-1}\,\prod_{0\leq a_2\leq a_1} \cdots \prod_{0\leq a_r\leq a_{r-1}} \left(X-(1+\xi_D^{a_1}+\cdots+\xi_D^{a_r})\right).
 \end{equation}
\end{corollary}

\begin{proof}
This is a direct consequence of the above theorem. 
\end{proof}

The linear recurrence given in Corollary \ref{gencoro} is not necessarily the minimal linear recurrence with integer coefficients satisfied by $\{S(n)\}$.  However, the characteristic polynomial of the 
minimal of such recurrences must be a factor of $P_S(X)$.

\begin{example}
 Let $F$ be a $n$-variable Boolean function.  The {\it nega-Hadamard transform of F} is defined as the complex valued function given by
\begin{equation}
\label{nega1}
 \mathcal{N}_{F}(\boldsymbol{a}) = \sum_{{\bf x} \in \mathbb{F}_2^n} (-1)^{F({\bf x})\oplus\boldsymbol{a}\cdot {\bf x}} \,\,i^{w({\bf x})},
\end{equation}
where $i=\sqrt{-1}$ and $w({\bf x})$ is the Hamming weight of the vector ${\bf x}$.  According to Riera and Parker \cite{rp}, the nega-Hadamard transform is central
to the structural analysis of pure $n$-qubit stabilizer quantum states.

Consider the case $\boldsymbol{a}=\boldsymbol{0}$, which is the equivalent of the exponential sum in this setting.  If $F$ is symmetric, then $\mathcal{N}_{F}(\boldsymbol{0})$ can be written as a 
binomial sum.  In particular,
\begin{equation}
 \mathcal{N}_{\boldsymbol{e}_{n,[k_1,\cdots,k_s]}}(\boldsymbol{0})=\sum_{q=0}^n \binom{n}{q} i^{q} (-1)^{\binom{q}{k_1}+\cdots+\binom{q}{k_s}}.
\end{equation}
Let $r=\lfloor\log_2(k_s)\rfloor+1$ and $D=2^{r}$.  Lucas' Theorem and Corollary \ref{corocloseda} imply that
\begin{equation}
\mathcal{N}_{\boldsymbol{e}_{n,[k_1,\cdots,k_s]}}(\boldsymbol{0}) = \sum_{j=0}^{D-1} \left(\frac{1}{D}\sum_{t=0}^{D-1}(-1)^{\binom{t}{k_1}+\cdots+\binom{t}{k_s}}\xi_D^{tj}\right)\lambda_j^n,
\end{equation}
where $\lambda_j = 1+i \xi_D^{-j}$.  Moreover, Corollary \ref{gencoro} implies that the sequence $\{\mathcal{N}_{\boldsymbol{e}_{n,[k_1,\cdots,k_s]}}(\boldsymbol{0})\}$ satisfies the linear
recurrence with integer coefficients given by
\begin{eqnarray}
 P(X)&=& \prod_{a=0}^{2^r-1} \left(X-(1+i \xi_D^a)\right)\\\nonumber
 &=& (X-2)\Phi_{4}(X-1)\Phi_8(X-1)\cdots \Phi_{2^r}(X-1),
\end{eqnarray}
where $\Phi_n(X)$ is the $n$-th cyclotomic polynomial.

We would like to point out that this is not a new result. It was already established in \cite{cms}.  However, we decided to include it because it is a straightforward application of our results.
 \end{example}

\begin{example}
 Consider the sum
 \begin{equation}
  S(n)=\sum _{q_1=0}^n \sum _{q_2=0}^{n-q_1}\sum _{q_3=0}^{n-q_1-q_2} \binom{n}{q_1}\binom{n-q_1}{q_2} \binom{n-q_1-q_2}{q_3}\xi_5^{q_1+q_2+q_3},
 \end{equation}
where $\xi_5=\exp(2\pi i/5)$.  Let $F(q_1,q_2,q_3)=q_1+q_2+q_3$.  Note that $F(q_1,q_2,q_3)\mod 5$ is clearly periodic in each component with period 5.  Therefore, Corollary \ref{gencoro} implies that
$\{S(n)\}$ satisfies the linear recurrence whose characteristic polynomial is given by
\begin{align}
 P_S(X)=&\prod _{a_1=0}^4 \prod _{a_2=0}^{a_1} \prod _{a_3=0}^{a_2}\left(X-\left(1+\xi_5^{a_1}+\xi_5^{a_2}+\xi_5^{a_3}\right)\right).
\end{align}
However, the minimal linear recurrence with integer coefficients satisfied by $\{S(n)\}$ has characteristic polynomial
\begin{eqnarray}
 \mu_S(X)&=&X^5-5 X^4+10 X^3-10 X^2+5 X-244\\\nonumber
 &=&(X-4) \left(X^4-X^3+6 X^2+14 X+61\right).
\end{eqnarray}
Thus, it must be true that $\mu_S(X)|P_S(X)$. Indeed, after simplification, we have
\begin{align}
 P_S(X)=& (X-4) \left(X^2-3 X+1\right) \left(X^4-11 X^3+46 X^2-86 X+61\right)\\\nonumber
 &\left(X^4-6 X^3+16 X^2-21 X+11\right) \left(X^4-6 X^3+16 X^2-16 X+16\right)\\\nonumber
 &\left(X^4-X^3-4 X^2+4 X+11\right) \left(X^4-X^3+X^2-X+1\right)\\\nonumber
 &\left(X^4-X^3+6 X^2-6 X+11\right) \left(X^4-X^3+6 X^2+4 X+1\right)\\\nonumber
 &\left(X^4-X^3+6 X^2+14
   X+61\right).
\end{align}
The fact that $\mu_S(X)|P_S(X)$ is now evident.
\end{example}

\begin{example}
Other toy examples can be constructed with previous classical results.  For example, it is known that $\{f_n \mod m\}$, where $f_n$ represents the $n$-th Fibonacci number and $m$ is a positive integer, is 
periodic.  The period is known as the Pisano period mod $m$ and it is usually denoted by $\pi(m)$.  Let $f_n^{(m)}$ represent $f_n\mod m$ and consider the sum
\begin{equation}
 S_m(n)=\sum_{q=0}^n \binom{n}{q}\xi_{\pi(m)}^{f_q^{(m)}},
\end{equation}
where $\xi_{\pi(m)}=\exp(2\pi i/\pi(m))$.  Corollary \ref{gencoro} implies that $\{S_m(n)\}$ satisfies the linear recurrence with integer coefficients whose characteristic polynomial is given by
\begin{equation}
 P_{S_m}(X) = \prod_{a=0}^{\pi(m)-1} (X-(1+\xi_{\pi(m)}^a)).
\end{equation}
Moreover, Corollary \ref{corocloseda} implies that its closed form is given by
\begin{equation}
 S_m(n)=\sum_{j=0}^{\pi(m)-1} \left(\frac{1}{\pi(m)}\sum_{t=0}^{\pi(m)-1}\xi_{\pi(m)}^{f_t^{(m)}+tj}\right) \left(1+\xi_{\pi(m)}^{-j}\right)^n.
\end{equation}
This example can be easily generalized to any Lucas sequence of the first kind $u_n(a,b)$ (the Fibonacci sequence is given by $u_n(1,-1)$).
\end{example}

Let us go back to our exponential sums.  The above results can be used to obtain closed formulas for exponential sums of elementary symmetric polynomials.  Let $\mathbb{F}_q=\{0,\alpha_1,\cdots,\alpha_{q-1}\}$.  Theorem \ref{expsumformthm} implies that
\begin{equation}
 S_{\mathbb{F}_q}(\boldsymbol{e}_{n,k}) = \sum_{m_1=0}^n\sum_{m_2=0}^{n-m_1}\cdots\sum_{m_{q-1}=0}^{n-m_1-\cdots-m_{q-1}}{n\choose m_0^*,m_1,m_2,\cdots, m_{q-1}}\xi_p^{\text{Tr}
 \left(\Lambda_{\alpha_1,\cdots,\alpha_{q-1}}\left(k,m_1,\cdots,m_{q-1}\right)\right)}
\end{equation}
where $m_0^*=n-(m_1+\cdots+m_{q-1})$, $\xi_p=\exp(2\pi i/p)$ and $\text{Tr}=\text{Tr}_{\mathbb{F}_q/\mathbb{F}_p}$.  Moreover, note that
$${n\choose m_0^*,m_1,m_2,\cdots, m_{q-1}}=\binom{n}{m_1}\binom{n-m_1}{m_2}\cdots\binom{n-m_1-\cdots-m_{q-2}}{m_{q-1}}.$$
Therefore, if we let 
\begin{equation}
F_{k;\mathbb{F}_q}(m_1,\cdots,m_{q-1})=\Lambda_{\alpha_1,\cdots,\alpha_{q-1}}(k,m_1,\cdots, m_{q-1}),
\end{equation}
then
\begin{equation}
\label{goodtype}
S_{\mathbb{F}_q}(\boldsymbol{e}_{n,k}) = \sum_{m_1=0}^n\sum_{m_2=0}^{n-m_1}\cdots\sum_{m_{q-1}=0}^{n-m_1-\cdots-m_{q-1}}{n\choose m_0^*,m_1,m_2,\cdots, m_{q-1}}\xi_p^{\text{Tr}
\left(F_{k,\mathbb{F}_q}\left(m_1,\cdots,m_{q-1}\right)\right)}
\end{equation}
is of the same type as (\ref{gensum}).  It remains to show the periodicity of $F_{k;\mathbb{F}_q}$.

We start with the following lemma.

\begin{lemma}
\label{lambdaperiod}
 Let $p$ be prime and $a_1,\cdots, a_l$ be some elements in some field extension of $\mathbb{F}_p$.  Define
 \begin{equation}
  \Lambda^{(p)}_{a_1,\cdots,a_l}(k,m_1,\cdots,m_l)= \Lambda_{a_1,\cdots,a_l}(k,m_1^{+},\cdots,m_l^{+})\mod p,
\end{equation}
where 
$$m_j^{+} = \begin{cases}
 m_j, & \text{ if }m_j>0 \\
 m_j + \left(\left\lfloor\frac{-m_j}{D}\right\rfloor+1\right)D, & \text{ if }m_j\leq 0.
\end{cases}
$$
Then, $\Lambda^{(p)}_{a_1,\cdots, a_l}(k,m_1,\cdots,m_l)$ is periodic in each of the variables $m_1,\cdots, m_l$ with period $D=p^{\lfloor\log_p(k)r\rfloor+1}$.
\end{lemma}

\begin{proof}
We first show that if $m_1,\cdots, m_l$ are all non-negative, then $$\Lambda_{a_1,\cdots,a_l}(k,m_1,\cdots,m_j+D,\cdots, m_l)\equiv\Lambda_{a_1,\cdots,a_l}(k,m_1,\cdots,m_j,\cdots, m_l)\mod p$$ 
for each $j=1,\cdots, l$.  The proof of this claim is by induction on $l$.  

Suppose first that $l=1$.  That is, consider
 \begin{equation}
  \Lambda_{a_1}(k,m_1) = a^k\binom{m_1}{k}.
 \end{equation}
Lucas' Theorem implies that if $D=p^{\lfloor\log_p(k)\rfloor+1}$, then
\begin{equation}
 \binom{m_1+D}{k'}\equiv \binom{m_1}{k'}\mod p,
\end{equation}
for ever $k'\leq k$.
Therefore, $\Lambda_{a_1}(k',m_1+D)\equiv \Lambda_{a_1}(k',m_1)\mod p$ for every $k'\leq k$ and the result holds for $l=1$.  

Suppose now that the result holds for some $l\geq 1$.  Consider $\Lambda_{a_1,\cdots,a_l,a_{l+1}}(k,m_1,\cdots,m_l,m_{l+1})$.  Recall that
\begin{equation}
\label{lambdaind}
 \Lambda_{a_1,\cdots,a_l,a_{l+1}}(k,m_1,\cdots,m_l,m_{l+1}) = \sum_{j=0}^{m_{l+1}} \binom{m_{l+1}}{j}a_{l+1}^j \Lambda_{a_1,\cdots, a_l}(k-j,m_1,\cdots, m_l).
\end{equation}
It is clear that $$\Lambda_{a_1,\cdots,a_l,a_{l+1}}(k,m_1,\cdots,m_j+D,\cdots,m_l,m_{l+1})\equiv \Lambda_{a_1,\cdots,a_l,a_{l+1}}(k,m_1,\cdots,m_j,\cdots,m_l,m_{l+1})\mod p$$ holds for $j=1,\cdots, l$ (induction hypothesis).  It remains to show that it is also true for the variable $m_{l+1}$.  In order to do that, first note that a simple induction argument shows that if $k<0$, then $$\Lambda_{a_1,\cdots, a_l}(k,m_1,\cdots, m_l)=0.$$
Therefore, every term on the right-hand side of (\ref{lambdaind}) for which $j>k$ is 0.  This implies that the binomial coefficient that accompanies every surviving term in (\ref{lambdaind}) satisfies (Lucas' Theorem)
\begin{equation}
 \binom{m_{l+1}+D}{j}\equiv \binom{m_{l+1}}{j} \mod p.
\end{equation}
Then,
\begin{eqnarray}\nonumber
 \Lambda_{a_1,\cdots, a_l,a_{l+1}}(k,m_1,\cdots,m_l,m_{l+1}+D) &=& \sum_{j=0}^{m_{l+1}+D}\binom{m_{l+1}+D}{j}a_{l+1}^j \Lambda_{a_1,\cdots,a_{l}}(k-j,m_1,\cdots,m_l)\\
 &\equiv& \sum_{j=0}^{m_{l+1}+D}\binom{m_{l+1}}{j}a_{l+1}^j \Lambda_{a_1,\cdots,a_{l}}(k-j,m_1,\cdots,m_l)\mod p\\\nonumber
 &\equiv& \sum_{j=0}^{m_{l+1}}\binom{m_{l+1}}{j}a_{l+1}^j \Lambda_{a_1,\cdots,a_{l}}(k-j,m_1,\cdots,m_l)\mod p\\\nonumber
 &\equiv& \Lambda_{a_1,\cdots,a_{l+1}}(k,m_1,\cdots,m_{l+1}) \mod p.
\end{eqnarray}
Therefore, 
$$\Lambda_{a_1,\cdots, a_{l+1}}(k,m_1,\cdots, m_{l+1}+D)\equiv \Lambda_{a_1,\cdots, a_{l+1}}(k,m_1,\cdots, m_{l+1})\mod p$$ is also true.   We conclude by induction that if $m_1,\cdots,m_l$ are 
non-negative integers,
then
$$\Lambda_{a_1,\cdots, a_l}(k,m_1,\cdots,m_j+D,\cdots,m_l)\equiv \Lambda_{a_1,\cdots, a_l}(k,m_1,\cdots,m_j,\cdots,m_l)\mod p$$
for $j=1,\cdots, l$ and $D=p^{\lfloor\log_p(k)r\rfloor+1}$.

It is clear that 
\begin{equation}
\Lambda_{a_1,\cdots, a_l}(k,m_1,\cdots,m_j+t D,\cdots,m_l)\equiv \Lambda_{a_1,\cdots, a_l}(k,m_1,\cdots,m_j,\cdots,m_l)\mod p
\end{equation}
for every non-negative integer $t$.  Sadly, the same cannot be said about negative $t$.  For example, if $m_l$ is negative, then by the inductive definition of $\Lambda_{a_1,\cdots, a_l}$ one has that $\Lambda_{a_1,\cdots,a_l}(k,m_1,\cdots, m_l)=0$.  However, this can be circumvented by defining the function 
$$\Lambda^{(p)}_{a_1,\cdots,a_l}(k,m_1,\cdots,m_l):= \Lambda_{a_1,\cdots,a_l}(k,m_1^{+},\cdots,m_l^{+})\mod p,$$
where 
\begin{equation}
m_j^{+} = \begin{cases}
 m_j, & \text{ if }m_j>0 \\
 m_j + \left(\left\lfloor\frac{-m_j}{D}\right\rfloor+1\right)D, & \text{ if }m_j\leq 0.
\end{cases}
\end{equation}
Observe that $$\Lambda^{(p)}(k,m_1,\cdots,m_j+t D,\cdots,m_l)=\Lambda^{(p)}(k,m_1,\cdots,m_j,\cdots,m_l)$$ for every $t\in \mathbb{Z}$ and $j=1,\cdots,l$.  In other words, $\Lambda^{(p)}_{a_1,\cdots,a_l}(k,m_1,\cdots,m_l)$ is periodic in each of the variables $m_1,\cdots,m_l$ with period $D$.  This concludes the proof.
\end{proof}

Let us go back to formula (\ref{goodtype}) for $S_{\mathbb{F}_q}(\boldsymbol{e}_{n,k})$.  Note that the value of $\xi_p^{\Tr(F_{k;\mathbb{F}_q}(m_1,\cdots,m_{q-1}))}$ depends only on the value of $F_{k;\mathbb{F}_q}(m_1,\cdots,m_l)\mod p$.  Therefore, if we define
\begin{equation}
F^{(p)}_{k;\mathbb{F}_q}(m_1,\cdots,m_{q-1}):= \Lambda^{(p)}_{\alpha_1,\cdots,\alpha_{q-1}}(k,m_1,\cdots,m_{q-1}),
\end{equation}
then
\begin{equation}
\label{goodtype2}
S_{\mathbb{F}_q}(\boldsymbol{e}_{n,k}) = \sum_{m_1=0}^n\sum_{m_2=0}^{n-m_1}\cdots\sum_{m_{q-1}=0}^{n-m_1-\cdots-m_{q-1}}{n\choose m_0^*,m_1,m_2,\cdots, m_{q-1}}\xi_p^{\text{Tr}
\left(F^{(p)}_{k,\mathbb{F}_q}\left(m_1,\cdots,m_{q-1}\right)\right)}.
\end{equation}

We now present our closed formulas for $S_{\mathbb{F}_q}(\boldsymbol{e}_{n,k})$.  This generalizes Cai et al.'s result for the binary case \cite{cai}.  It also generalizes the recurrence exploited in \cite{cm1,cm2}.
\begin{theorem}
\label{closedformsSq}
 Let $n$ and $k>1$ be positive integers and $p$ be a prime and $q=p^r$ with $r\geq 1$.  Let $D=p^{\lfloor\log_p(k)\rfloor+1}$. Then,
 $$ S_{\mathbb{F}_q}(\boldsymbol{e}_{n,k}) = \sum_{j_1=0}^{D-1}\sum_{j_2=0}^{j_1}\cdots \sum_{j_{q-1}=0}^{j_{q-2}} 
 c_{j_1,\cdots,j_{q-1}}(k)\left(1+\xi_D^{-j_1}+\cdots+\xi_D^{-j_{q-1}}\right)^n,$$
 where 
 \begin{equation*}
  c_{j_1,\cdots,j_{q-1}}(k)=\frac{1}{D^{q-1}}\sum_{b_{q-1}=0}^{D-1}\cdots \sum_{b_1=0}^{D-1} \xi_p^{\Tr\left(F_{k;\mathbb{F}_q}^{(p)}\left(b_1,\cdots,b_{q-1}\right)\right)}
  \sum_{(j_1',\cdots,j_{q-1}')\in \Sym(j_1,\cdots, j_{q-1})}
 \xi_D^{j_1'b_{q-1}+\cdots+j_{q-1}' b_1},
 \end{equation*}
$\xi_{m}=\exp(2\pi i/m)$, $\Tr=\Tr_{\mathbb{F}_q/\mathbb{F}_p}$, and $\lambda_{j_1,\cdots, j_{q-1}} = 1+\xi_D^{-j_1}+\xi_{D}^{-j_2}+\cdots+\xi_D^{-j_{q-1}}$.
In particular, the sequence $\{S_{\mathbb{F}_q}(\boldsymbol{e}_{n,k})\}$ satisfies the linear recurrence with integer coefficients whose characteristic polynomial is given by
$$P_{q,k}(X)=\prod_{a_1=0}^{D-1}\,\prod_{0\leq a_2\leq a_1} \cdots \prod_{0\leq a_{q-1}\leq a_{q-2}} \left(X-\left(1+\xi_D^{a_1}+\cdots+\xi_D^{a_{q-1}}\right)\right).$$
\end{theorem}

\begin{proof}
The sum in (\ref{goodtype2}) is of type (\ref{gensum}).  Moreover, Lemma \ref{lambdaperiod} implies that $F^{(p)}_{n,k;\mathbb{F}_q}(m_1,\cdots,m_{q-1})$ is periodic in each component with period $D$.  
The result now follows from Theorem \ref{generalclosedA} and its corollary.
\end{proof}

Theorem \ref{closedformsSq} also provides a bound for the degree of the minimal linear recurrence with integer coefficients satisfied by $\{S_{\mathbb{F}_q}(\boldsymbol{e}_{n,k})\}$.

\begin{corollary}
 Let $k>1$ be positive integers and $p$ be a prime and $q=p^r$ with $r\geq 1$.  Let $D=p^{\lfloor\log_p(k)\rfloor+1}$. 
 The degree of the minimal linear recurrence with integer coefficients that $\{S_{\mathbb{F}_q}(\boldsymbol{e}_{n,k})\}$ satisfies is less than or equal to $(D)_q/q!$, where 
 $(a)_n = a(a+1)(a+2)\cdots(a+n-1)$ is the Pochhammer symbol.
\end{corollary}

\begin{proof}
 The characteristic polynomial of such recurrence is a factor of $P_{q,k}(X)$.  The result now follows from the fact that the degree of $P_{q,k}(X)$ is $(D)_q/q!$.
\end{proof}

\begin{example}
\label{exF4}
 Consider the sequence $\{S_{\mathbb{F}_4}(\boldsymbol{e}_{n,3})\}$.  Theorem \ref{closedformsSq} implies that this sequence satisfies the linear recurrence whose characteristic is given by
\begin{eqnarray*}
 P_{4,3}(X)&=& \prod_{a_1=0}^3\prod_{0\leq a_2\leq a_1}\prod_{0\leq a_3\leq a_2}\left(X-\left(1+i^{a_1}+i^{a_2}+i^{a_3}\right)\right)\\
 &=&(X-4) (X-2)^2 X^2 (X+2) \left(X^2+4\right) \left(X^2-6 X+10\right) \left(X^2-4 X+8\right)\\
 &&\left(X^2-2 X+2\right)^2 \left(X^2-2 X+10\right) \left(X^2+2 X+2\right).
\end{eqnarray*}
The minimal linear recurrence with integer coefficients that $\{S_{\mathbb{F}_4}(\boldsymbol{e}_{n,3})\}$ satisfies has characteristic polynomial given by
$$\mu_{4,3}(X)=(X-4) (X-2) \left(X^2+4\right).$$
Note that, as expected, $\mu_{4,3}(X)|P_{4,3}(X)$.  After simplification, the closed formula given by Theorem \ref{closedformsSq} is
\begin{eqnarray*}
S_{\mathbb{F}_4}(\boldsymbol{e}_{n,3})&=&4^{n-1}+3\cdot 2^{n-1}-\frac{3}{4}(2i)^n-\frac{3}{4}(-2i)^n\\
&=&4^{n-1}+3\cdot 2^{n-1}-3\cdot 2^{n-1} \cos \left(\frac{n\pi}{2}\right).
\end{eqnarray*}
The function $\Tr_{\mathbb{F}_4/\mathbb{F}_2}(\boldsymbol{e}_{n,3})$ can be identified with a $2n$-variable Boolean function.  The identification depends on the value-vector of $\Tr_{\mathbb{F}_4/\mathbb{F}_2}(\boldsymbol{e}_{n,3})$,
which is a $2n$-tuple of 0's and 1's, and an order of the elements of $\mathbb{F}_{2}^{2n}$ (different order, different representation).  For instance, $\Tr_{\mathbb{F}_4/\mathbb{F}_2}(\boldsymbol{e}_{4,3})$ can be identified with
\begin{align*}
 F_8({\bf X})=&X_2 X_3 X_5+X_1 X_4 X_5+X_2 X_4 X_5+X_2 X_7 X_5+X_4 X_7 X_5+X_1 X_8 X_5+X_2 X_8 X_5+\\
 &X_3 X_8 X_5+X_4 X_8 X_5+X_1 X_3 X_6+X_2 X_3 X_6+X_1 X_4 X_6+X_2 X_3 X_7+X_1 X_4 X_7+\\
 &X_2 X_4 X_7+X_1 X_6 X_7+X_2 X_6 X_7+X_3 X_6 X_7+X_4 X_6 X_7+X_1 X_3 X_8+X_2 X_3 X_8+\\
 &X_1 X_4 X_8+X_1 X_6 X_8+X_3 X_6 X_8.
\end{align*}
Observe that $S_{\mathbb{F}_4}(\boldsymbol{e}_{4,3})=S_{\mathbb{F}_2}(F_8)=64.$
\end{example}

\begin{example}
 Consider the sequence $\{S_{\mathbb{F}_8}(\boldsymbol{e}_{n,3})\}$.  Theorem \ref{closedformsSq} implies that this sequence satisfies the linear recurrence whose characteristic is given by
\begin{eqnarray*}
 P_{8,3}(X)&=& \prod_{a_1=0}^3\prod_{a_2=0}^{a_1}\prod_{a_3=0}^{a_2}\prod_{a_4=0}^{a_3}\prod_{a_5=0}^{a_4}\prod_{a_6=0}^{a_5}\prod_{a_7=0}^{a_6}\left(X-\left(1+i^{a_1}+i^{a_2}+i^{a_3}+i^{a_4}+i^{a_5}+i^{a_6}+i^{a_7}\right)\right).
\end{eqnarray*}
The minimal linear recurrence with integer coefficients that $\{S_{\mathbb{F}_8}(\boldsymbol{e}_{n,3})\}$ satisfies has characteristic polynomial given by
$$\mu_{8,3}(X)=(X-4) (X+4) \left(X^2+16\right) \left(X^2-8 X+32\right) \left(X^2-4 X+8\right) \left(X^2+4 X+8\right).$$
It can be verified that $\mu_{8,3}(X)|P_{8,3}(X)$.  The closed formula for this exponential sum is given (after simplification) by
\begin{equation*}
 S_{\mathbb{F}_8}(\boldsymbol{e}_{n,3})=\frac{1}{8} \left(2 \sqrt{2}\right)^n \left(\left(9+(-1)^n\right) \left(\sqrt{2}\right)^n+2 \left(2^n+9\right) \sin \left(\frac{n\pi}{4}\right)
 -6 \sin \left(\frac{3n \pi}{4}\right)-6 \left(\sqrt{2}\right)^n \cos \left(\frac{n\pi}{2}\right)\right).
\end{equation*}
As with the previous example, the function $\Tr_{\mathbb{F}_8/\mathbb{F}_2}(\boldsymbol{e}_{n,3})$ can be identified with a $3n$-variable Boolean function.
\end{example}

These two examples show a big difference between the degrees of the polynomials $P_{q,k}(X)$ and $\mu_{q,k}(X)$, where $\mu_{q,k}(X)$ represents the characteristic polynomial of the minimal linear recurrence
with integer coefficients satisfied by the sequence $\{S_{\mathbb{F}_q}(\boldsymbol{e}_{n,k})\}$.  In particular, $P_{q,k}(X)$ does not seem to be tight.  However, what you are seeing here
is the fact that when working over $\mathbb{F}_q$ with $q=p^r$ and $r>1$, some of the factors of $P_{q,k}(X)$ are repeated multiple times.  For instance, consider Example \ref{exF4}.  Observe that when
$(a_1,a_2,a_3)=(2,1,0)$ we get the factor $X-(1 + i)$.  However when $(a_1,a_2,a_3)=(3,1,1)$, we also get the factor $X-(1+i)$. Therefore, this factor is repeated twice. The factor $X-(1-i)$ is also repeated twice.  That is why the factor
$X^2-2 X+2$ appears in $P_{4,3}(X)$ with 2 as exponent.  This phenomenon does not occur over $\mathbb{F}_p$.  In fact, there are examples where the polynomial $P_{p,k}(X)$ is tight.

\begin{example}
 Consider the sequence $\{S_{\mathbb{F}_3}(\boldsymbol{e}_{n,7})\}$.  The characteristic polynomial of the minimal linear recurrence with integer coefficients satisfied by 
 this sequence is 
 $$\mu_{3,7}(X)=\frac{1}{X}P_{3,7}(X).$$
The term $1/X$ in front of $P_{3,7}(X)$ comes from the fact that $P_{3,7}(0)=0$, i.e., 0 is a root for $P_{3,7}(X)$.  However, the root 0 does not contribute anything to the closed formula for the exponential sum.
Therefore, taking the term $X$ does not alter the result.  Thus, the polynomial $P_{3,7}(X)$ is tight for this example.
\end{example}

The repetition of factors can be eliminated by using {\it least common multiples} ($\lcm$).

\begin{theorem}
 Let $n$ and $k>1$ be positive integers and $p$ be a prime and $q=p^r$ with $r\geq 1$.  Let $D=p^{\lfloor\log_p(k)\rfloor+1}$.  Let $M_{a_1,\cdots,a_{q-1}}(X)$ be the minimal polynomial
 for the algebraic integer $1+\xi_D^{a_1}+\cdots+\xi_D^{a_{q-1}}$.  Then, $\{S_{\mathbb{F}_q}(\boldsymbol{e}_{n,k})\}$ satisfies the linear recurrence with integer coefficients whose characteristic
 polynomial is given by
 $$\chi_{q,k}(X)=\lcm\left(\mu_{a_1,\cdots,a_{q-1}}(X)\right)_{0\leq a_{q-1}\leq \cdots \leq a_2\leq a_1 \leq D-1}.$$
\end{theorem}

We point out that Theorem \ref{closedformsSq} and other results after it can be extended to linear combinations of elementary symmetric polynomials without too much effort. For instance, 
suppose that $0\leq k_1<\cdots<k_s$ are integers and 
$\beta_1,\cdots,\beta_s \in \mathbb{F}_q^{\times}$. The discussion prior Theorem \ref{closedformsSq} together with Corollary \ref{expsumformcoro} implies that
\begin{align}
\label{goodtype3}
S_{\mathbb{F}_q}\left(\sum_{j=1}^s \beta_j\boldsymbol{e}_{n,k_j}\right) =& \sum_{m_1=0}^n\sum_{m_2=0}^{n-m_1}\cdots\sum_{m_{q-1}=0}^{n-m_1-\cdots-m_{q-1}}{n\choose m_0^*,m_1,m_2,\cdots, m_{q-1}}\\\nonumber
& \times \xi_p^{\text{Tr}\left(\sum_{j=1}^s \beta_j F^{(p)}_{k,\mathbb{F}_q}\left(m_1,\cdots,m_{q-1}\right)\right)}.
\end{align}
The statement of Theorem \ref{closedformsSq} can now be written almost verbatim for linear combinations of elementary symmetric polynomials.
The only differences are that $D$ is now $D=p^{\lfloor\log_p(k_s)\rfloor+1}$
and $$\Tr\left(F_{k;\mathbb{F}_q}^{(p)}\left(b_1,\cdots,b_{q-1}\right)\right)$$ in the definition of $c_{j_1,\cdots,j_{q-1}}(k)$ must be replaced by
$$\Tr\left(\sum_{j=1}^s \beta_j F_{k_j;\mathbb{F}_q}^{(p)}\left(b_1,\cdots,b_{q-1}\right)\right).$$
Similar adjustments apply to the other results.

\section{Concluding remarks}
We expressed exponential sums of symmetric polynomials over finite fields as multinomial sums.  These expressions represent a computational improvement over the definition of exponential sums.  
These expressions also provided a link between balancedness of symmetric polynomials over Galois fields and a problem similar to the one of bisecting binomial coefficients.  We also proved closed 
formulas for exponential sums of symmetric polynomials over Galois fields by exploiting their multinomial sum representations.  These closed formulas extend the work of Cai, Green and Thierauf on the 
binary field to every finite field.  Moreover, we showed that the recursive nature of these exponential sums is not special to the binary case.  Finally, since every multi-variable function over a finite 
field extension of $\mathbb{F}_2$ can be identified with a Boolean function, then perhaps these results can be used to find new families of Boolean functions that might be useful for efficient implementations. \medskip

\noindent
{\bf Acknowledgments.}  The authors would like to thank Oscar E. Gonz\'alez for reading a previous version of this article. His comments and suggestions improve the presentation of this work.

\bibliographystyle{plain}

\end{document}